\documentclass[12pt,a4paper]{article}
\usepackage{amsmath,amssymb,amsthm}
\usepackage[margin=1in]{geometry}
\usepackage[colorlinks=true,linkcolor=blue,urlcolor=blue,citecolor=blue]{hyperref}

\newtheorem{theorem}{Theorem}[section]
\newtheorem{lemma}[theorem]{Lemma}

\theoremstyle{remark}
\newtheorem{remark}[theorem]{Remark}

\newcommand{\R}{\mathbb R}
\newcommand{\norm}[1]{\left\lVert#1\right\rVert_2}

\title{Doubly exponential convergence of the cyclic steepest descent method for strictly convex quadratics in arbitrary dimensions}
\author{Ran Gu\thanks{NITFID, School of Statistics and Data Science, LPMC, KLMDASR, LEBPS and AAIS, Nankai University, Tianjin 300071, People's Republic of China. E-mail: rgu@nankai.edu.cn.}}
\date{\today}

\begin{document}
\maketitle

\begin{abstract}
We study the cyclic steepest descent method (CSD) for strictly convex quadratic minimization.
CSD repeats, for $j$ consecutive iterations, the exact steepest-descent step size computed
at the beginning of each cycle.
The only rigorous result for the real algorithm has so far been restricted to two
dimensions, where the cycle-starting gradient sequence is known to converge doubly
exponentially.
For the simplified (``simple'') model obtained by discarding bounded logarithmic terms,
Dai and Fletcher predicted that CSD is superlinear whenever the number $n$ of distinct
eigenvalues is below twice the cycle length $m$, and linear otherwise.
Let $A$ be symmetric positive definite with $q$ distinct eigenvalues, and suppose
$2j>q$. We prove the superlinear side of this threshold for the real algorithm in arbitrary
dimensions, and in fact obtain a faster, doubly exponential, decay.
Except for a Lebesgue-null set of initial points, for every $\kappa$ below an explicit
positive threshold $\kappa_0$ (given in Theorem~\ref{thm:full}), the cycle-starting
gradient satisfies
\[
 \norm{g_k}\le \exp(-C e^{\kappa k}),
\]
and the full gradient and iterate errors satisfy analogous bounds with exponent $\kappa/j$
after $m$ iterations.
This is the first rigorous proof for the real CSD in arbitrary dimensions, including
repeated eigenvalues, and confirms the threshold $2j>q$ (the simple-model prediction
$n<2m$); the doubly exponential rate is faster than the superlinear one predicted by the
simple model.
The proof combines arbitrary-reference ratio coordinates, inverse-image volume contraction,
thin-band estimates, and a per-component Borel--Cantelli argument.
\end{abstract}


\section{Introduction}
Consider the unconstrained minimization problem
\[
\min_{u\in\mathbb R^d} f(u),
\]
where $f$ is continuously differentiable.
Gradient methods take the form $u_{m+1}=u_m-t_m\nabla f(u_m)$ and, because they require
only first-order information and little storage, remain central to large-scale optimization.
The classical steepest descent (SD) method of Cauchy~\cite{cauchy1847} chooses the step
$t_m^{SD}=h_m^\top h_m/(h_m^\top A h_m)$, with $h_m=\nabla f(u_m)$, that minimizes
$f(u_m-t h_m)$ along the line.
For a strictly convex quadratic $f(u)=\frac12 u^\top A u+b^\top u$ with $A=A^\top\succ0$,
the error decays geometrically~\cite{akaike1959}, that is, $R$-linearly, and exhibits the
well-known zigzag phenomenon; the rate is directly governed by the condition number and
is very slow when the condition number is large.

To overcome this limitation, Barzilai and Borwein introduced the two-point BB
method~\cite{barzilai1988}, which abandons exact line search and constructs a step size from
the gradient displacement of the previous iterate, thereby acquiring an approximate
quasi-Newton property.
In the quadratic case, the two BB step sizes are reciprocal Rayleigh quotients of the
current gradient with respect to $A$; the method is $R$-superlinearly convergent in two
dimensions~\cite{barzilai1988} and $R$-linearly convergent in general
dimensions~\cite{dailiao2002}, and performs considerably better than SD in practice.
A large family of step-size methods followed.
Friedlander, Mart\'inez, Molina, and Raydan introduced the class of
``gradient methods with retards''~\cite{friedlander1999}: the exact SD step size computed
at the current gradient is reused for a fixed number of consecutive iterations, and the
cyclic steepest descent method (CSD) belongs to this class.
Dai and Yuan proposed the alternate minimization (AM) gradient method, equivalent to the
CaBB method independently introduced by Raydan and Svaiter~\cite{dayuan2003,raydan2002},
in which the step size alternately minimizes the function value and the gradient norm along
the steepest-descent line; they proved it to be $Q$-superlinearly convergent in two
dimensions and $Q$-linearly convergent in arbitrary dimensions.
Fletcher later developed the limited-memory steepest descent method
(LMSD)~\cite{fletcher2012}, which estimates the spectrum of the Hessian from Krylov-subspace
information over the last few gradients and terminates a $d$-dimensional quadratic problem
in $2d-1$ steps.

A parallel line of work seeks finite termination:
Yuan proposed a step size that terminates a two-dimensional quadratic problem in three
steps~\cite{yuan2006}, and Xie, Liu, Sun, and Yuan recently extended this to a step size that
terminates a three-dimensional quadratic problem in five steps, in cyclic
form~\cite{xie2026}.
In contrast to constructing a ``perfect step size'' for finite termination, we return to
the simplest cyclic SD step and study its asymptotic decay rate in arbitrary dimensions.

Regarding the convergence of CSD itself, a clear gap separates rigorous results from
asymptotic predictions.
Dai and Fletcher systematically studied the asymptotic behaviour of CSD and the BB
method~\cite{daifletcher2005}: by taking logarithms of the individual gradient components,
ordering them by magnitude, and discarding bounded logarithmic terms in the recurrence,
they obtained a simplified (simple) model independent of the eigenvalues of $A$;
for this model they proved that CSD is $m$-step superlinear when the effective dimension $n$
is below twice the cycle length $m$, and degenerates to linear convergence when $n\ge 2m$.
Gu and Du later used the same asymptotic framework for LMSD and its
modification~\cite{gudu2021}, showing that the superlinear threshold can be relaxed from
$O(m)$ to $O(m^2)$.
It should be emphasized that these rigorous theorems concern only the simplified model:
Dai and Fletcher explicitly state that a rigorous proof for the real algorithm in arbitrary
dimensions has not yet been established, and one can only infer the behaviour from numerical
observations and comparison with the simplified model.

Rigorous results for the real algorithm are still confined to special cases.
Shen, Li, and Dai first returned to the real CSD~\cite{shen-v1} and proved that, in two
dimensions, the cycle-starting gradient sequence $\{g_k\}$ converges doubly exponentially
to zero for almost every starting point.
This is the first rigorous result for the real CSD, but it is confined to two dimensions;
whether the real CSD in arbitrary dimensions obeys the threshold predicted by the simple
analysis has remained open for more than two decades.

For convenience, we denote by $q$ the number of distinct eigenvalues and by $j$ the cycle
length, corresponding respectively to the effective dimension $n$ and the cycle length $m$
of Dai and Fletcher.
We prove that this prediction holds for the real CSD, and is in fact stronger.
When $A$ has $q$ distinct eigenvalues and the cycle length $j$ satisfies
\begin{equation}
2j>q,
\end{equation}
except on a set of initial points of Lebesgue measure zero, the entire actual iteration
sequence of CSD$(j)$ converges doubly exponentially:
the gradient decays like $\exp(-C e^{\kappa k})$, where $\kappa$ may approach an explicitly
given critical value $\kappa_0$ (defined in \eqref{eq:kappa0-def}; notation as in
Section~\ref{sec:notation}).
The main contributions are as follows:
\begin{itemize}
  \item We establish the first rigorous convergence theorem for the real CSD (rather than
        its simplified model) in arbitrary dimensions, including repeated eigenvalues,
        extending the two-dimensional doubly exponential result of Shen, Li, and Dai to
        arbitrary dimensions.
  \item We rigorously verify that the threshold $2j>q$ (corresponding to $n<2m$ in the
        analysis of Dai and Fletcher) coincides with the simple-model prediction, and show
        that the positivity condition on the convergence exponent is the same inequality.
        Moreover, the real algorithm converges doubly exponentially, a faster rate than
        the superlinear one predicted by the simple model.
  \item The per-coordinate basis, inverse-image volume contraction, and thin-band estimates
        developed here may serve as a template for the rigorous convergence analysis of
        other cyclic or retarded gradient methods.
\end{itemize}

The remainder of this paper is organized as follows.
Section~\ref{sec:notation} states the problem, algorithm, and notation.
We then establish invertibility of the cycle map and inverse-image volume contraction in
log-ratio coordinates, followed by measure-theoretic preliminaries, thin-band estimates,
and a per-component Borel--Cantelli estimate.
Finally, these pieces combine into the main theorem, with repeated eigenvalues and the
actual iteration steps handled at the end.

\section{Problem, algorithm, and notation}\label{sec:notation}
Consider the $d$-dimensional strictly convex quadratic problem
\[
 f(u)=\tfrac12 u^\top A u+b^\top u,\qquad
 A=A^\top\succ0,\qquad u^*=-A^{-1}b.
\]
Fix a positive integer cycle length $j$.
CSD$(j)$ starts from an initial point $u_0$ and, for each cycle $k$, sets
\[
 u_{jk,0}=u_{jk},
\]
and performs $j$ consecutive steps with the same step size $t_k$:
\begin{equation}
 u_{jk+s+1}=u_{jk+s}-t_k\bigl(Au_{jk+s}+b\bigr),
 \qquad s=0,1,\dots,j-1.
 \label{eq:u-iteration}
\end{equation}
The step size is determined by the gradient at the beginning of the cycle:
\[
 t_k=\frac{g_k^\top g_k}{g_k^\top A g_k},\qquad
 g_k=Au_{jk}+b.
\]
Denote the gradient at the $m$-th actual iterate by
\[
 h_m=Au_m+b,
\]
so that the gradient at the beginning of a cycle satisfies $g_k=h_{jk}$.
The cycle index $k$ and the actual iteration index $m$ are related by
\[
 m=jk+s,\qquad 0\le s<j.
\]
From \eqref{eq:u-iteration} and $h_m=Au_m+b$, the gradients satisfy
\[
 h_{jk+s}=(I-t_kA)^s g_k\quad(0\le s\le j),\qquad
 g_{k+1}=(I-t_kA)^j g_k.
\]
The iteration stops when the gradient vanishes, after which all subsequent gradients are
set to zero.
We say that a sequence $\{g_k\}$ converges to zero doubly exponentially if there are
constants $\kappa>0$ and $C>0$ such that
\[
 \norm{g_k}\le \exp(-C e^{\kappa k})
\]
for all sufficiently large $k$, equivalently,
\[
 \liminf_{k\to\infty}\frac1k\log\bigl(-\log\norm{g_k}\bigr)\ge \kappa>0.
\]
Doubly exponential convergence is stronger than $R$-superlinear convergence and faster
than any polynomial-order $R$-convergence: for every $p>0$,
\[
 \limsup_{k\to\infty}\norm{g_k}^{1/k^p}=0.
\]

\begin{theorem}[Doubly exponential convergence in arbitrary dimensions]
\label{thm:full}
Let $A\in\R^{d\times d}$ be a fixed symmetric positive definite matrix with $q$ distinct
eigenvalues. Assume $q\ge2$ and fix an integer cycle length $j$ such that
\[
 2j>q.
\]
Define
\begin{equation}
 \kappa_0:=
 \begin{cases}
  \log(2j-1), & q=2,\\[6pt]
  \min\!\left\{\dfrac{1}{q-1}\log\dfrac{2j-1}{q-1},\ \dfrac{1}{q-2}\right\}, & q\ge3.
 \end{cases}
 \label{eq:kappa0-def}
\end{equation}
Then $\kappa_0>0$, and there is a set $\mathcal E_{A,j}$ of Lebesgue measure zero in the
initial-point space such that, for every $u_0\notin\mathcal E_{A,j}$ and every
$0<\kappa<\kappa_0$, there is $C=C(\kappa,u_0)>0$ for which CSD$(j)$ satisfies
\[
 \norm{g_k}\le \exp(-C e^{\kappa k}),\qquad
 \norm{h_m}\le \exp(-C e^{(\kappa/j)m}),\qquad
 \norm{u_m-u^*}\le \exp(-C e^{(\kappa/j)m})
\]
for all sufficiently large $k,m$.
If $q=1$, the method terminates in one step.
\end{theorem}

\section{Spectral projections, normalization, and arbitrary-reference ratio coordinates}
\subsection{Spectral projections and the radial dynamics}
Assume throughout that $q\ge2$ and $2j>q$.
Let the distinct eigenvalues of $A$ be
\[
 \lambda_1>\lambda_2>\cdots>\lambda_q>0,
\]
with associated orthogonal spectral projections $\mathcal P_1,\dots,\mathcal P_q$, and
let $d_i$ be the dimension of the $i$-th eigenspace, so that $\sum_{i=1}^q d_i=d$.
Define the gradient norm in the $i$-th eigenspace by
\[
 v_{i,k}=\norm{\mathcal P_i g_k},\qquad
 \norm{g_k}^2=\sum_{i=1}^q v_{i,k}^2.
\]
When the eigenvalues are simple and $q=d$, $v_{i,k}=|g_k^{(i)}|$.

Since $g_{k+1}=(I-t_kA)^j g_k$ and the spectral projections are orthogonal,
\[
 \mathcal P_i g_{k+1}=\mathcal P_i (I-t_kA)^j g_k
 =(1-t_k\lambda_i)^j \mathcal P_i g_k.
\]
Taking norms gives
\begin{equation}
 v_{i,k+1}=|1-t_k\lambda_i|^j v_{i,k}.
 \label{eq:radial-dynamics}
\end{equation}
The step size is
\[
 t_k=\frac{\sum_i v_{i,k}^2}{\sum_i\lambda_i v_{i,k}^2}.
\]
We first consider initial points with all $v_i>0$; points with zero projections and singular
backward orbits are handled at the end of the proof.

\subsection{Normalization and the scalar $\mu$}
Set
\[
 \Delta=\lambda_1-\lambda_q>0,\qquad
 \beta_i=\frac{\lambda_i-\lambda_q}{\Delta}\in[0,1],\qquad
 c=\frac{\lambda_q}{\Delta}>0.
\]
Then
\[
 1=\beta_1>\beta_2>\cdots>\beta_q=0.
\]
Define the normalized Rayleigh quotient
\begin{equation}
 \mu=\frac{\sum_{i=1}^q \beta_i v_i^2}{\sum_{i=1}^q v_i^2}.
 \label{eq:mu-def}
\end{equation}
Since all $v_i^2>0$ and at least one $\beta_i>0$ (e.g.\ $\beta_1=1$), the numerator is
positive and hence $\mu>0$.
Since $\beta_i\le1$ and at least one $\beta_i<1$ (e.g.\ $\beta_q=0$), the numerator is
strictly smaller than the denominator and hence $\mu<1$.
Thus
\[
 0<\mu<1.
\]
From the expression for $t_k$,
\begin{equation}
 t=\frac{\sum_i v_i^2}{\sum_i\lambda_i v_i^2}
 =\frac{\sum_i v_i^2}{\lambda_q\sum_i v_i^2+\Delta\sum_i\beta_i v_i^2}
 =\frac{1}{\lambda_q+\Delta\mu},
 \label{eq:t-mu}
\end{equation}
\begin{equation}
 1-t\lambda_i
 =\frac{\lambda_q+\Delta\mu-\lambda_i}{\lambda_q+\Delta\mu}
 =\frac{\Delta(\mu-\beta_i)}{\lambda_q+\Delta\mu}
 =\frac{\mu-\beta_i}{c+\mu}.
 \label{eq:factor-mu}
\end{equation}
Hence
\begin{equation}
 |1-t\lambda_i|=\frac{|\mu-\beta_i|}{c+\mu}.
 \label{eq:abs-factor}
\end{equation}
For completeness, note that $\beta_i=(\lambda_i-\lambda_q)/\Delta$ gives
$\lambda_i=\lambda_q+\Delta\beta_i$, and therefore
\[
 \lambda_q+\Delta\mu-\lambda_i=\Delta(\mu-\beta_i).
\]
Since $c=\lambda_q/\Delta$,
\[
 \frac{\Delta(\mu-\beta_i)}{\lambda_q+\Delta\mu}
 =\frac{\mu-\beta_i}{\lambda_q/\Delta+\mu}
 =\frac{\mu-\beta_i}{c+\mu},
\]
which completes the verification of \eqref{eq:factor-mu}.

\subsection{Ratio coordinates with an arbitrary reference $r$}
Fix an arbitrary reference index $r\in\{1,\dots,q\}$, set
\[
 I_r=\{1,\dots,q\}\setminus\{r\},\qquad p=q-1,
\]
and define
\[
 R_i=\left(\frac{v_i}{v_r}\right)^2,\qquad
 z_i=\log R_i\quad(i\in I_r).
\]
In these coordinates,
\begin{equation}
 B(z)=1+\sum_{i\in I_r}e^{z_i},\qquad
 \mu(z)=\frac{\beta_r+\sum_{i\in I_r}\beta_i e^{z_i}}{B(z)}.
 \label{eq:mu-chart}
\end{equation}
Indeed, dividing numerator and denominator of \eqref{eq:mu-def} by $v_r^2$ gives
\[
 \mu=\frac{\beta_r+\sum_{i\ne r}\beta_i R_i}{1+\sum_{i\ne r}R_i}.
\]
Each ratio updates as
\[
 R_i'
 =\left(\frac{v_i'}{v_r'}\right)^2
 =\left(\frac{|1-t\lambda_i|^j v_i}{|1-t\lambda_r|^j v_r}\right)^2
 =R_i\left|\frac{1-t\lambda_i}{1-t\lambda_r}\right|^{2j}
 =R_i\left|\frac{\mu-\beta_i}{\mu-\beta_r}\right|^{2j}.
\]
The last step uses \eqref{eq:factor-mu}:
\[
 \frac{1-t\lambda_i}{1-t\lambda_r}
 =\frac{(\mu-\beta_i)/(c+\mu)}{(\mu-\beta_r)/(c+\mu)}
 =\frac{\mu-\beta_i}{\mu-\beta_r}.
\]
Taking logarithms,
\begin{equation}
 (T_r(z))_i
 =z_i+2j\log\frac{|\mu-\beta_i|}{|\mu-\beta_r|}
 =z_i+2j\log|\mu-\beta_i|-2j\log|\mu-\beta_r|.
 \label{eq:arbitrary-chart}
\end{equation}
More explicitly, since $z_i=\log R_i$ and $z_i'=\log R_i'$,
\[
 z_i'=\log R_i'
 =\log R_i+2j\log\left|\frac{\mu-\beta_i}{\mu-\beta_r}\right|
 =z_i+2j\log|\mu-\beta_i|-2j\log|\mu-\beta_r|.
\]
When $r$ is an interior index, $\mu-\beta_r$ may change sign, and the version with absolute
values is proved below.

\section{Measure-theoretic preliminaries}

\begin{lemma}[Lebesgue null sets]
\label{lem:lebesgue-zero}
In $\R^p$, a Lebesgue null set has the following properties:
\begin{itemize}
 \item the countable union of null sets is null;
 \item a smooth diffeomorphism maps null sets to null sets;
 \item a smooth hypersurface is a null set.
\end{itemize}
\end{lemma}

\begin{lemma}[Inverse function theorem]
\label{lem:inverse-function}
Let $U\subset\R^p$ be open and $F:U\to\R^p$ a $C^k$ map ($k\ge1$).
If $z_0\in U$ and $\det DF(z_0)\ne0$, there are neighborhoods $V\subset U$ of $z_0$ and
$W$ of $F(z_0)$ such that $F|_V:V\to W$ is a $C^k$ diffeomorphism with
\[
 D(F|_V)^{-1}(F(z))=\bigl(DF(z)\bigr)^{-1},\qquad z\in V.
\]
In particular, if $\det DF(z)\ne0$ throughout $U$, then $F$ is a local $C^k$ diffeomorphism.
\end{lemma}

\begin{lemma}[Borel--Cantelli lemma]
\label{lem:bc}
If $\{E_n\}_{n=1}^\infty$ is a sequence of Lebesgue measurable sets in $\R^p$ with
\[
 \sum_{n=1}^\infty m_p(E_n)<\infty,
\]
then almost every point belongs to only finitely many $E_n$.
That is, outside a null set, for each $z_0$ there is an integer $\mathcal N$ such that
$z_0\notin E_n$ whenever $n>\mathcal N$.
\end{lemma}

\begin{lemma}[Fubini theorem]
\label{lem:fubini}
Let $(X,\mathcal A,\mu)$ and $(Y,\mathcal B,\nu)$ be $\sigma$-finite measure spaces.
For a nonnegative measurable $f:X\times Y\to[0,\infty]$,
\[
 \int_{X\times Y} f\,d(\mu\times\nu)
 =\int_X\left(\int_Y f\,d\nu\right)d\mu
 =\int_Y\left(\int_X f\,d\mu\right)d\nu.
\]
If $f$ is real-valued and $\int_{X\times Y}|f|\,d(\mu\times\nu)<\infty$, both iterated
integrals exist and agree.
\end{lemma}

\begin{lemma}[Fubini null-set lemma]
\label{lem:fubini-null}
If $Z\subset\R^p$ is a Lebesgue null set and $E\subset\R$ is Lebesgue measurable, then
$Z\times E$ is a Lebesgue null set in $\R^{p+1}$.
\end{lemma}
\begin{proof}
Take $X=\R^p$, $Y=\R$, $\mu=m_p$, $\nu=m_1$, and the indicator
$f=\mathbf 1_{Z\times E}:\R^{p+1}\to[0,1]$. By the nonnegative case of Fubini's theorem
(Lemma~\ref{lem:fubini}),
\[
 m_{p+1}(Z\times E)
 =\int_{\R}\left(\int_{\R^p}\mathbf 1_Z(x)\mathbf 1_E(t)\,dx\right)dt.
\]
For each fixed $t\in E$, the inner integral is
\[
 \int_{\R^p}\mathbf 1_Z(x)\,dx=m_p(Z)=0;
\]
for $t\notin E$ the inner integral is also zero. Hence
\[
 m_{p+1}(Z\times E)=\int_E 0\,dt=0.
\]
This completes the proof.
\end{proof}

\section{Inverse-image volume formula for an arbitrary reference}
Define the singular set
\[
 \Sigma_r=\bigcup_{s=2}^{q-1}\{z:\mu(z)=\beta_s\},
 \qquad \Omega_r=\R^p\setminus\Sigma_r.
\]
This union is empty when $q=2$. Since all spectral weights are positive, $\mu$ never
equals $0$ or $1$, and $\Omega_r$ is precisely the region in which all component update
factors are nonzero.

\begin{lemma}[Volume formula independent of the reference component]
\label{lem:volume}
On $\Omega_r$, $T_r$ is a smooth local diffeomorphism and
\[
 \det DT_r=1-2j.
\]
Every target $y\in\R^p$ has exactly $q-1$ preimages in $\Omega_r$, and the corresponding
values of $\mu$ lie in
\[
 \beta_{s+1}<\mu<\beta_s,\qquad s=1,\dots,q-1.
\]
Writing $m_p$ for $p$-dimensional Lebesgue measure, for every measurable $E\subset\R^p$,
\begin{equation}
 m_p(T_r^{-n}E)=\rho^n m_p(E),\qquad
 \rho=\frac{q-1}{2j-1}<1.
 \label{eq:inverse-volume}
\end{equation}
Here $T_r^{-n}E$ is understood as the union over the $q-1$ inverse branches, and
$\Sigma_r$, being a finite union of smooth level sets, is closed and measurable; the
backward images below are taken along each well-defined inverse branch.
The set
\[
 \mathcal N_r=\Sigma_r\cup\bigcup_{n\ge1}(T_r|_{\Omega_r})^{-n}\Sigma_r
\]
of initial points that hit the singular set in finite steps is a countable union of null
sets and hence is null.
\end{lemma}

\begin{proof}
\textbf{Step 1: $\Sigma_r$ is null.}
For $l\in I_r$, from \eqref{eq:mu-chart},
\[
 \mu = \frac{\beta_r+\sum_{i\in I_r}\beta_i e^{z_i}}{1+\sum_{i\in I_r}e^{z_i}}.
\]
Differentiating with respect to $z_l$,
\[
 \frac{\partial \mu}{\partial z_l}
 =\frac{e^{z_l}(\beta_l-\mu)}{B(z)}.
\]
Since $0<\mu<1$, we have $1-\mu>0$ and $-\mu<0$.
If $r\ne1$, then $1\in I_r$; taking $l=1$,
\[
 \frac{\partial \mu}{\partial z_1}
 =\frac{e^{z_1}(1-\mu)}{B(z)}>0.
\]
If $r=1$, then $q\in I_r$; taking $l=q$,
\[
 \frac{\partial \mu}{\partial z_q}
 =\frac{e^{z_q}(0-\mu)}{B(z)}<0.
\]
Thus $\nabla\mu\ne0$.
Each singular level set $\{\mu=\beta_s\}$ is a smooth hypersurface, and a finite union
of smooth hypersurfaces is null; hence $m_p(\Sigma_r)=0$.

\textbf{Step 2: Jacobian determinant.}
Let
\[
 f_i(\mu)=2j\log\frac{|\mu-\beta_i|}{|\mu-\beta_r|},\qquad i\in I_r.
\]
Where $\mu\ne\beta_i,\beta_r$,
\[
 f_i'(\mu)
 =\frac{2j(\beta_i-\beta_r)}{(\mu-\beta_i)(\mu-\beta_r)}.
\]
Then
\[
 DT_r=I+uv^\top,
\]
where
\[
 u_i=f_i'(\mu),\qquad v_l=\mu_{z_l}.
\]
By the matrix determinant lemma,
\[
 \det(I+uv^\top)=1+v^\top u.
\]
We now compute $v^\top u$. Differentiating \eqref{eq:mu-chart} with respect to $z_l$,
\[
 v_l=\frac{\partial \mu}{\partial z_l}
 =\frac{e^{z_l}(\beta_l-\mu)}{B(z)}.
\]
Therefore
\begin{align*}
 v^\top u
 &=\sum_{l\in I_r} f_l'(\mu)\,\frac{\partial\mu}{\partial z_l}\\
 &=\sum_{l\in I_r}
 \frac{2j(\beta_l-\beta_r)}{(\mu-\beta_l)(\mu-\beta_r)}
 \cdot\frac{e^{z_l}(\beta_l-\mu)}{B(z)}.
\end{align*}
For each $l\in I_r$, since $\beta_l-\mu=-(\mu-\beta_l)$,
\[
 \frac{\beta_l-\mu}{\mu-\beta_l}=-1.
\]
Substituting,
\[
 v^\top u
 =-\frac{2j}{(\mu-\beta_r)B(z)}
 \sum_{l\in I_r}(\beta_l-\beta_r)e^{z_l}.
\]
From \eqref{eq:mu-chart},
\[
 B(z)(\mu-\beta_r)
 =\sum_{l\in I_r}(\beta_l-\beta_r)e^{z_l}.
\]
Substituting this into the previous expression,
\[
 v^\top u
 =-\frac{2j}{(\mu-\beta_r)B(z)}\cdot B(z)(\mu-\beta_r)
 =-2j.
\]
Hence
\[
 \det DT_r=1+v^\top u=1-2j,
\]
and therefore $|\det DT_r|=2j-1\ne0$.
Since $\mu(z)$ is $C^\infty$ on $\Omega_r$ and $\mu(z)\ne\beta_i,\beta_r$ for
$i\in I_r$, the map $T_r$ is $C^\infty$ on $\Omega_r$. By the inverse function theorem
(Lemma~\ref{lem:inverse-function}), $T_r$ is a smooth local diffeomorphism on $\Omega_r$.

\textbf{Step 3: Number of preimages and inverse branches.}
Given a target point $y=(y_i)_{i\in I_r}$, let $\tau=\mu(z)$ be the scalar value at a
preimage. Then
\[
 e^{y_i}=e^{z_i}\left|\frac{\tau-\beta_i}{\tau-\beta_r}\right|^{2j}
 =R_i\left|\frac{\tau-\beta_i}{\tau-\beta_r}\right|^{2j}.
\]
Since $2j$ is even and $\tau\ne\beta_i,\beta_r$,
\[
 \left|\frac{\tau-\beta_i}{\tau-\beta_r}\right|^{2j}
 =\frac{(\tau-\beta_i)^{2j}}{(\tau-\beta_r)^{2j}},
\]
and hence
\[
 R_i=e^{y_i}\left|\frac{\tau-\beta_r}{\tau-\beta_i}\right|^{2j}
 =e^{y_i}\frac{(\tau-\beta_r)^{2j}}{(\tau-\beta_i)^{2j}}.
\]
Substituting into the definition $\tau=\mu(z)$,
\[
 \tau=\frac{\beta_r+\sum_{i\in I_r}\beta_i R_i}{1+\sum_{i\in I_r}R_i},
\]
so that
\[
 \tau\Bigl(1+\sum_{i\in I_r}R_i\Bigr)=\beta_r+\sum_{i\in I_r}\beta_i R_i,
\]
or
\[
 (\tau-\beta_r)+\sum_{i\in I_r}(\tau-\beta_i)R_i=0.
\]
Substituting the expression for $R_i$,
\[
 (\tau-\beta_r)+\sum_{i\in I_r}(\tau-\beta_i)
 e^{y_i}\frac{(\tau-\beta_r)^{2j}}{(\tau-\beta_i)^{2j}}=0.
\]
Factoring out $(\tau-\beta_r)$ and using $\tau\ne\beta_r$,
\[
 (\tau-\beta_r)\left[
 1+\sum_{i\in I_r}e^{y_i}
 \frac{(\tau-\beta_r)^{2j-1}}{(\tau-\beta_i)^{2j-1}}
 \right]=0.
\]
Therefore
\[
 1+\sum_{i\in I_r}e^{y_i}
 \frac{(\tau-\beta_r)^{2j-1}}{(\tau-\beta_i)^{2j-1}}=0.
\]
Dividing both sides by $(\tau-\beta_r)^{2j-1}$,
\[
 \frac{1}{(\tau-\beta_r)^{2j-1}}
 +\sum_{i\in I_r}\frac{e^{y_i}}{(\tau-\beta_i)^{2j-1}}=0.
\]
Since $2j-1$ is odd,
\[
 \frac{1}{(\tau-\beta_r)^{2j-1}}=-\frac{1}{(\beta_r-\tau)^{2j-1}},
 \qquad
 \frac{1}{(\tau-\beta_i)^{2j-1}}=-\frac{1}{(\beta_i-\tau)^{2j-1}}.
\]
Substituting and multiplying both sides by $-1$, we obtain
\[
 G_y(\tau):=
 \frac{1}{(\beta_r-\tau)^{2j-1}}
 +\sum_{i\in I_r}\frac{e^{y_i}}{(\beta_i-\tau)^{2j-1}}=0.
\]
Now consider the intervals
\[
 (0,\beta_{q-1}),\ (\beta_{q-1},\beta_{q-2}),\ \dots,\ (\beta_2,1).
\]
On any such interval $I=(\beta_{s+1},\beta_s)$, $\tau$ equals no $\beta_i$, so all
denominators $\beta_i-\tau$ have fixed sign. Differentiating $G_y$,
\[
 G_y'(\tau)=
 \frac{2j-1}{(\beta_r-\tau)^{2j}}
 +\sum_{i\in I_r}\frac{(2j-1)e^{y_i}}{(\beta_i-\tau)^{2j}}>0,
\]
since $2j$ is even and all denominators are positive. Thus $G_y$ is strictly increasing
on $I$.
As $\tau$ approaches the left endpoint $\beta_{s+1}$ from inside the interval:
if $\beta_{s+1}\ne\beta_r$, then for some $l\in I_r$, $\beta_l=\beta_{s+1}$, and
\[
 \beta_l-\tau=\beta_{s+1}-\tau\to0^-,
\]
while $2j-1$ is odd, so
\[
 \frac{e^{y_l}}{(\beta_l-\tau)^{2j-1}}\to-\infty.
\]
The remaining terms stay bounded as $\tau$ is sufficiently close to $\beta_{s+1}$, so
$G_y(\tau)\to-\infty$.
If $\beta_{s+1}=\beta_r$, then the first term
\[
 \frac{1}{(\beta_r-\tau)^{2j-1}}\to-\infty,
\]
the remaining terms stay bounded, and again $G_y(\tau)\to-\infty$.
As $\tau$ approaches the right endpoint $\beta_s$ from inside the interval:
if $\beta_s\ne\beta_r$, then for some $l\in I_r$, $\beta_l=\beta_s$, and
\[
 \beta_l-\tau=\beta_s-\tau\to0^+,
\]
while $2j-1$ is odd, so
\[
 \frac{e^{y_l}}{(\beta_l-\tau)^{2j-1}}\to+\infty.
\]
The remaining terms stay bounded, so $G_y(\tau)\to+\infty$.
If $\beta_s=\beta_r$, then the first term
\[
 \frac{1}{(\beta_r-\tau)^{2j-1}}\to+\infty,
\]
the remaining terms stay bounded, and again $G_y(\tau)\to+\infty$.
By the intermediate value theorem, $G_y$ has exactly one root in each interval.
There are $q-1$ intervals in total, hence exactly $q-1$ roots.
Each root gives a preimage via
\[
 R_i=e^{y_i}\frac{(\tau-\beta_r)^{2j}}{(\tau-\beta_i)^{2j}}>0.
\]
Each inverse branch is smooth, with absolute Jacobian determinant $1/(2j-1)$.

\textbf{Step 4: Inverse-image volume.}
Applying the change-of-variables formula on all $q-1$ inverse branches and summing,
\[
 m_p(T_r^{-1}E)=\frac{q-1}{2j-1}m_p(E).
\]
Iterating gives \eqref{eq:inverse-volume}. The condition $2j>q$ ensures
\[
 \rho=\frac{q-1}{2j-1}<1.
\]
Finally, applying the volume formula to the null set $\Sigma_r$ and taking a countable
union yields $m_p(\mathcal N_r)=0$.
\end{proof}

\section{Thin-band estimate with arbitrary signs}

\begin{lemma}[Thin-band volume estimate]
\label{lem:thin-band}
Let
\[
 H(z)=b+\sum_{l=1}^p a_l e^{z_l},\qquad b\ne0,
\]
where the $a_l$ may be positive, negative, or zero.
Let $\mathcal B\subset\R^p$ be an axis-parallel box of side length $L>0$, i.e.\
\[
 \mathcal B=[c_1,c_1+L]\times\cdots\times[c_p,c_p+L]
\]
for some $c_1,\dots,c_p\in\R$.
For $0<\varepsilon<|b|/2$,
\[
 m_p\bigl(\mathcal B\cap\{|H|<\varepsilon\}\bigr)
 \le\frac{4p^2}{|b|}\varepsilon L^{p-1}.
\]
\end{lemma}

\begin{proof}
Let
\[
 J=\{l:a_lb<0\}
\]
be the set of indices whose coefficients have sign opposite to the constant term.
If $J$ is empty, every nonzero $a_l$ has the same sign as $b$, so each term $a_l e^{z_l}$
has the same sign as $b$, no cancellation occurs in the sum, and
\[
 |H|=|b|+\sum_{l:a_l\ne0}|a_l|e^{z_l}\ge |b|>\varepsilon,
\]
so the set to be estimated is empty and the claim is trivial.
Assume henceforth that $J\ne\varnothing$.
On the set $\{|H|<\varepsilon\}$, write
\[
 H=\underbrace{b+\sum_{l\notin J}a_l e^{z_l}}_{H_1}
 +\underbrace{\sum_{l\in J}a_l e^{z_l}}_{H_2}.
\]
For $l\notin J$, $a_lb\ge0$, so $a_l e^{z_l}$ has the same sign as $b$ (or is zero), and
\[
 H_1=b+\sum_{l\notin J}a_l e^{z_l}
\]
has the same sign as $b$, with
\[
 |H_1|=|b|+\sum_{l\notin J}|a_l|e^{z_l}\ge |b|.
\]
Moreover,
\[
 |H_2|=\left|\sum_{l\in J}a_l e^{z_l}\right|
 =\sum_{l\in J}|a_l|e^{z_l}.
\]
By the triangle inequality,
\[
 |H|\ge |H_1|-|H_2|
 \ge |b|-\sum_{l\in J}|a_l|e^{z_l}.
\]
Since $|H|<\varepsilon$,
\[
 |b|-\sum_{l\in J}|a_l|e^{z_l}<\varepsilon,
\]
i.e.
\[
 \sum_{l\in J}|a_l|e^{z_l}
 >|b|-\varepsilon>\frac{|b|}{2}.
\]
Therefore at least one $l\in J$ satisfies
\[
 |a_l|e^{z_l}\ge\frac{|b|}{2p}.
\]
Indeed, if $|a_l|e^{z_l}<|b|/(2p)$ held for every $l\in J$, then
\[
 \sum_{l\in J}|a_l|e^{z_l}<|J|\frac{|b|}{2p}\le p\cdot\frac{|b|}{2p}=\frac{|b|}{2},
\]
a contradiction.
Cover the set to be estimated by at most $p$ sets:
\[
 \mathcal B\cap\{|H|<\varepsilon\}
 \subset
 \bigcup_{l\in J}
 \left(\mathcal B\cap\{|H|<\varepsilon\}\cap
 \{|a_l|e^{z_l}\ge |b|/(2p)\}\right).
\]
Fix $l\in J$. On the set
\[
 \mathcal S_l:=
 \mathcal B\cap\{|H|<\varepsilon\}\cap
 \{|a_l|e^{z_l}\ge |b|/(2p)\},
\]
fix all coordinates except $z_l$. Then $H$ as a function of $z_l$ is
\[
 H=b+\sum_{j\ne l}a_j e^{z_j}+a_l e^{z_l},
\]
with derivative
\[
 \frac{\partial H}{\partial z_l}=a_l e^{z_l}.
\]
The restriction $|a_l|e^{z_l}\ge |b|/(2p)>0$ implies $a_l\ne0$, and
\[
 \frac{\partial H}{\partial z_l}=a_l e^{z_l}
\]
does not change sign on the entire cross-section. Hence $H$ is strictly monotone in $z_l$.
When $z_l$ varies by $\Delta z_l$, the change in function value is at least
\[
 \left|\frac{\partial H}{\partial z_l}\right|\Delta z_l
 \ge \frac{|b|}{2p}\Delta z_l.
\]
The $z_l$-interval satisfying $|H|<\varepsilon$ corresponds to a function-value interval
contained in $(-\varepsilon,\varepsilon)$, of length at most $2\varepsilon$, so
\[
 \frac{|b|}{2p}\Delta z_l\le 2\varepsilon,
\]
i.e.
\[
 \Delta z_l\le \frac{4p\varepsilon}{|b|}.
\]
The projection of the remaining $p-1$ coordinates onto the box $\mathcal B$ has volume at
most $L^{p-1}$. By Fubini's theorem (Lemma~\ref{lem:fubini}),
\begin{align*}
 m_p(\mathcal S_l)
 &=\int_{\R^{p-1}}\left(\int_{\R}
 \mathbf 1_{\mathcal S_l}(z_1,\dots,z_p)\,dz_l\right)
 d(z_j)_{j\ne l}\\
 &=\int_{\R^{p-1}}\Delta z_l\,d(z_j)_{j\ne l}\\
 &\le\frac{4p\varepsilon}{|b|}\cdot L^{p-1}.
\end{align*}
There are at most $p$ such sets $\mathcal S_l$, hence
\[
 m_p\bigl(\mathcal B\cap\{|H|<\varepsilon\}\bigr)
 \le p\cdot\frac{4p\varepsilon}{|b|}L^{p-1}
 =\frac{4p^2}{|b|}\varepsilon L^{p-1}.
\]
This completes the proof.
\end{proof}

\begin{remark}[Why every reference works]
Fix $r$ and define
\begin{equation}
 \begin{split}
 D_r(z)&=\sum_{l\in I_r}(\beta_l-\beta_r)e^{z_l}
       =B(z)(\mu-\beta_r),\\
 H_i^{[r]}(z)&=(\beta_r-\beta_i)
       +\sum_{l\in I_r}(\beta_l-\beta_i)e^{z_l}
       =B(z)(\mu-\beta_i),\qquad i\in I_r.
 \end{split}
 \label{eq:H-D}
\end{equation}
Every $H_i^{[r]}$ has nonzero constant term $\beta_r-\beta_i$.
Hence Lemma~\ref{lem:thin-band} applies for every reference $r$ and every $i\ne r$.
\end{remark}

\section{Per-component Borel--Cantelli estimate}

\begin{lemma}[The logarithmic sum of every eigenspace component diverges exponentially]
\label{lem:all-averages}
Define
\[
 \rho=\frac{q-1}{2j-1}\in(0,1),\qquad p=q-1,
\]
with $\kappa_0$ as in \eqref{eq:kappa0-def}.
Then $\kappa_0>0$, and for each fixed $0<\kappa<\kappa_0$ and each reference
$r\in\{1,\dots,q\}$, there is a null set $Z_r(\kappa)\supset\mathcal N_r$ such that for
every $z_0\notin Z_r(\kappa)$,
\begin{equation}
 \frac{1}{e^{\kappa n}}\sum_{k=0}^{n-1}-\log|\mu_k-\beta_r|
 \longrightarrow+\infty.
 \label{eq:average-r}
\end{equation}
\end{lemma}

\begin{proof}
Fix $0<\kappa<\kappa_0$. On orbits avoiding $\mathcal N_r$, set
\[
 \ell_k^{[r]}=-\log|\mu_k-\beta_r|>0,\qquad
 S_n^{[r]}=\sum_{k=0}^{n-1}\ell_k^{[r]}.
\]
Fix positive integers $R,M$ and define the bad event
\[
 E_{n,r}(R,M,\kappa)
 =\{z_0\in[-R,R]^p\setminus\mathcal N_r:
 S_n^{[r]}\le M e^{\kappa n}\}.
\]
Take $z_0\in E_{n,r}(R,M,\kappa)$, i.e.\ $S_n^{[r]}\le M e^{\kappa n}$.
Since $\ell_k^{[r]}>0$, for any $0\le k\le n$,
\[
 S_k^{[r]}\le S_n^{[r]}\le M e^{\kappa n}.
\]
From \eqref{eq:arbitrary-chart},
\[
 z_{i,k+1}
 =z_{i,k}+2j\log|\mu_k-\beta_i|-2j\log|\mu_k-\beta_r|.
\]
Since $|\mu_k-\beta_i|\le1$,
\[
 z_{i,k+1}\le z_{i,k}+2j\ell_k^{[r]}.
\]
Summing,
\[
 z_{i,k}\le z_{i,0}+2jS_k^{[r]}
 \le R+2jM e^{\kappa n}=:U_n.
\]
Thus all coordinates have the upper bound $U_n$.

We next estimate the lower bound. From \eqref{eq:H-D},
\[
 D_r(z_k)=B(z_k)(\mu_k-\beta_r),\qquad
 H_i^{[r]}(z_k)=B(z_k)(\mu_k-\beta_i).
\]
Since $B(z_k)\le (p+1)e^{U_n}$,
\[
 \log|D_r(z_k)|\le U_n+\log(p+1).
\]
If at step $k$ all $i\ne r$ satisfy
\[
 |H_i^{[r]}(z_k)|\ge e^{-n},
\]
then
\[
 z_{i,k+1}
 \ge z_{i,k}-2j(n+U_n+\log(p+1)).
\]
If the condition has held at all steps up to $k$, then
\[
 z_{i,k}
 \ge -R-2j n(n+U_n+\log(p+1)).
\]
Set
\[
 V_n=R+2j n(n+U_n+\log(p+1)),\qquad
 \mathcal B_n=[-V_n,U_n]^p.
\]
If the condition held for all $n$ steps, then $z_n\in\mathcal B_n$, i.e.\
$z_0\in T_r^{-n}\mathcal B_n$.
Otherwise, let $k<n$ be the first step at which the condition is violated; then
$z_k\in\mathcal B_n$, and for some $i\ne r$,
\[
 |H_i^{[r]}(z_k)|<e^{-n}.
\]
Hence
\[
 z_k\in\mathcal C_n
 :=\mathcal B_n\cap
 \bigcup_{i\in I_r}\{|H_i^{[r]}|<e^{-n}\}.
\]
Therefore
\begin{equation}
 E_{n,r}(R,M,\kappa)\subset
 T_r^{-n}\mathcal B_n\ \cup\
 \bigcup_{k=0}^{n-1}T_r^{-k}\mathcal C_n.
 \label{eq:bad-event-cover}
\end{equation}

We estimate the volumes.
Since $U_n=O(e^{\kappa n})$ and $V_n=O(n e^{\kappa n})$, the side length of $\mathcal B_n$
is $O(n e^{\kappa n})$, so
\[
 m_p(\mathcal B_n)=O(n^p e^{\kappa p n}).
\]
For $\mathcal C_n$, apply the thin-band Lemma~\ref{lem:thin-band} with
$\varepsilon=e^{-n}$ and box side length $L=O(n e^{\kappa n})$:
\[
 m_p(\mathcal C_n)=O(n^{p-1} e^{\kappa (p-1) n} e^{-n}).
\]
By Lemma~\ref{lem:volume} and \eqref{eq:bad-event-cover},
\begin{align*}
 m_p(E_{n,r}(R,M,\kappa))
 &\le \rho^n m_p(\mathcal B_n)
 +\sum_{k=0}^{n-1}\rho^k m_p(\mathcal C_n)\\
 &\le C \left(n^p e^{\kappa p n}\rho^n
 +n^{p-1} e^{\kappa (p-1) n} e^{-n}\right),
\end{align*}
where $\rho=(q-1)/(2j-1)<1$. Throughout, the implicit constants depend only on
$A,j,R,M,\kappa$.

We first note $\kappa_0>0$. Since $2j>q$ and $q\ge2$, we have $2j-1\ge q\ge2$, hence
$\rho<1$ and $-\log\rho>0$. When $q=2$, $\kappa_0=\log(2j-1)>0$. When $q\ge3$,
$\kappa_0$ is the smaller of two positive numbers, so $\kappa_0>0$.

When $q=2$, $p=1$ and $\kappa_0=\log(2j-1)$, so
\[
 e^{\kappa p}\rho
 =e^{\kappa}\rho
 <e^{\kappa_0}\rho
 =(2j-1)\cdot\frac{1}{2j-1}=1.
\]
The second term $e^{\kappa(p-1)n}e^{-n}=e^{-n}$ decays exponentially on its own.

When $q\ge3$, $p\ge2$, and $\kappa_0$ is the smaller of two numbers:
\[
 \kappa_0\le\frac{1}{p}\log\frac1\rho
 \implies e^{\kappa_0 p}\rho\le1,
\]
\[
 \kappa_0\le\frac{1}{p-1}
 \implies e^{\kappa_0(p-1)}e^{-1}\le1.
\]
Since $\kappa<\kappa_0$, both inequalities are strict:
\[
 e^{\kappa p}\rho<1,\qquad
 e^{\kappa(p-1)}e^{-1}<1.
\]
Therefore both series
\[
 \sum_{n=1}^\infty n^p (e^{\kappa p}\rho)^n,\qquad
 \sum_{n=1}^\infty n^{p-1} (e^{\kappa(p-1)}e^{-1})^n
\]
converge, and hence
\[
 \sum_{n=1}^\infty m_p(E_{n,r}(R,M,\kappa))<\infty.
\]
By the Borel--Cantelli lemma, for almost every starting point in the box, eventually
$S_n^{[r]}>M e^{\kappa n}$ always holds.
Taking a countable union over $R,M$ and adjoining $\mathcal N_r$ yields the null set
$Z_r(\kappa)$ and \eqref{eq:average-r}.
\end{proof}

\section{From individual components to the whole actual iteration sequence}
\begin{proof}[Proof of Theorem~\ref{thm:full}]
Fix $\kappa_0$ as in \eqref{eq:kappa0-def} and set
\[
 \kappa_N=\kappa_0\left(1-\frac1N\right),\qquad N=2,3,\dots
\]
so $0<\kappa_N<\kappa_0$ and $\kappa_N\uparrow\kappa_0$.

We first handle the exceptional set in the amplitude space
$v_0=(v_{1,0},\dots,v_{q,0})\in(0,\infty)^q$.
For each $r$ and each $N$, Lemma~\ref{lem:all-averages} gives a null set
$Z_r(\kappa_N)\subset\R^p$. The map
\[
 \Phi_r:(0,\infty)\times\R^p\longrightarrow(0,\infty)^q,
 \qquad v_r=s,\quad v_i=s e^{z_i/2}\ (i\ne r)
\]
is a smooth diffeomorphism. By Fubini's theorem, $(0,\infty)\times Z_r(\kappa_N)$ is a null
set in $q$ dimensions; since $\Phi_r$ is a smooth diffeomorphism and maps null sets to null
sets, its image is also null in amplitude space.
Taking a countable union over $r=1,\dots,q$ and $N=2,3,\dots$ gives a null set
$\mathcal V$ in amplitude space.
Thus, for every $v_0\notin\mathcal V$, for all $r$ and all $N$,
\[
 \frac{S_n^{[r]}}{e^{\kappa_N n}}\to+\infty,
\]
and the orbit avoids all singular sets on which a component vanishes.

Now fix $v_0\notin\mathcal V$ and any $0<\kappa<\kappa_0$.
Choose $N$ large enough that $\kappa<\kappa_N$. Since
\[
 \frac{S_n^{[r]}}{e^{\kappa_N n}}\to+\infty,
\]
we have
\[
 \frac{S_n^{[r]}}{e^{\kappa n}}
 =\frac{S_n^{[r]}}{e^{\kappa_N n}}\,e^{(\kappa_N-\kappa)n}
 \to+\infty.
\]
Therefore for each $r$,
\begin{equation}
 S_n^{[r]}\ge C_r e^{\kappa n}
 \label{eq:lower-S}
\end{equation}
for all sufficiently large $n$.

From \eqref{eq:radial-dynamics} and \eqref{eq:abs-factor}, for each $r$,
\begin{align*}
 v_{r,n}
 &=v_{r,0}\exp\left(j\sum_{k=0}^{n-1}\log|\mu_k-\beta_r|
 -j\sum_{k=0}^{n-1}\log(c+\mu_k)\right)\\
 &=v_{r,0}\exp\left(-j S_n^{[r]}
 -j\sum_{k=0}^{n-1}\log(c+\mu_k)\right).
\end{align*}
Since $c+\mu_k\in[c,c+1]$,
\[
 -j\sum_{k=0}^{n-1}\log(c+\mu_k)=O(n).
\]
Combining with \eqref{eq:lower-S},
\[
 v_{r,n}\le \exp(-C'_r e^{\kappa n})
\]
for all sufficiently large $n$. Hence
\[
 \norm{g_n}\le q^{1/2}\max_{1\le r\le q}v_{r,n}
 \le \exp(-C e^{\kappa n})
\]
for all sufficiently large $n$.

We now lift the null set in amplitude space to the original gradient space to cover the
case of repeated eigenvalues.
Let the $i$-th eigenspace have dimension $d_i$, with $\sum_i d_i=d$.
In each region with nonzero projection, write in block polar coordinates
\[
 \mathcal P_i g_0=v_{i,0}\omega_i,\qquad
 v_{i,0}>0,\quad \omega_i\in S^{d_i-1}.
\]
Since $I-t_kA$ acts as a scalar $1-t_k\lambda_i$ on the $i$-th eigenspace, the direction
$\omega_i$ is preserved along the iteration; only the radius $v_{i,k}$ evolves, so tracking
the $v_i$'s suffices.
In these coordinates the Lebesgue volume element is
\[
 \prod_{i=1}^q v_{i,0}^{d_i-1}\,dv_{i,0}\,d\sigma_i(\omega_i).
\]
The function $v\mapsto \prod_i v_i^{d_i-1}$ is locally integrable on $(0,\infty)^q$.
Since $\mathcal V$ is a Lebesgue null set in the amplitude variables,
\[
 \int_{\mathcal V}\prod_i v_i^{d_i-1}\,dv=0.
\]
By the Fubini--Tonelli theorem,
\[
 \int_{\prod_i S^{d_i-1}}\int_{\mathcal V}
 \prod_i v_i^{d_i-1}\,dv\,d\sigma=0.
\]
Hence the corresponding set in the original gradient space is also null.
When $d_i=1$, the angular variable is only a sign and is treated by finite counting.
Any set on which some projection is zero is a finite union of proper linear subspaces and
is likewise null. Finally, through the invertible affine map
\[
 \Psi:\R^d\to\R^d,\qquad \Psi(u_0)=Au_0+b,
\]
we obtain the exceptional set in the initial-point space. This map is an affine
homeomorphism and maps Lebesgue null sets to null sets. Thus, if
$\mathcal G\subset\R^d$ is null in gradient space, then
\[
 \mathcal E_{A,j}=\Psi^{-1}(\mathcal G)
\]
is null in the initial-point space.

To obtain the conclusion for every actual iterate, set
\[
 K=\max\{1,\lambda_1/\lambda_q-1\}.
\]
The Rayleigh quotient bounds for the step size give
\[
 \frac1{\lambda_1}\le t_k\le \frac1{\lambda_q}.
\]
Since $I-t_kA$ has eigenvalues $1-t_k\lambda_i$, their absolute values are maximized at the
extreme eigenvalues:
\[
 \max_i|1-t_k\lambda_i|\le
 \max\left\{\left|1-\frac{\lambda_1}{\lambda_q}\right|,\left|1-\frac{\lambda_q}{\lambda_1}\right|\right\}
 =\frac{\lambda_1}{\lambda_q}-1\le K.
\]
Hence
\[
 \norm{I-t_kA}\le K.
\]
For $m=jk+s$ with $0\le s<j$,
\[
 \norm{h_m}
 =\norm{(I-t_kA)^s g_k}
 \le K^{j-1}\norm{g_k}.
\]
Hence
\[
 \norm{h_m}
 \le K^{j-1}\exp(-C e^{\kappa k}).
\]
Since $m=jk+s$, we have $k\ge m/j-1$, and therefore
\[
 e^{\kappa k}\ge e^{\kappa(m/j-1)}=e^{-\kappa}e^{\kappa m/j}.
\]
Consequently,
\[
 \norm{h_m}
 \le \exp(-C' e^{(\kappa/j)m})
\]
for all sufficiently large $m$, where $C'=C e^{-\kappa}$.
Since $u_m-u^*=A^{-1}h_m$,
\[
 \norm{u_m-u^*}
 \le \norm{A^{-1}}\norm{h_m}
 \le \exp(-C'' e^{(\kappa/j)m})
\]
for all sufficiently large $m$.
If $q=1$, then $A=\lambda_1I$ and the exact line search reaches the optimum in one step.
This completes the proof.
\end{proof}

\section*{Acknowledgements}
This work was supported by the National Key R\&D Program of China under Grant
No.~2022YFA1003800, the National Natural Science Foundation of China under Grant
No.~12671364, and the Natural Science Foundation of Tianjin under Grant
No.~25JCJQJC00300.

\bibliographystyle{plain}
\bibliography{cyclicSD_refs}

\end{document}